\documentclass[10pt]{amsart}

\usepackage{graphicx}
\usepackage{amsmath}
\usepackage{amssymb}
\usepackage{mathrsfs} 
\usepackage{hyperref}
\usepackage{enumerate}

\usepackage[utf8]{inputenc}

\newtheorem{thm}{Theorem}[section]

\newtheorem{lem}[thm]{Lemma}
\newtheorem{cor}[thm]{Corollary}
\newtheorem{con}[thm]{Conjecture}

\theoremstyle{definition}
\newtheorem{definition}[thm]{Definition}
\newtheorem{example}[thm]{Example}

\theoremstyle{remark}

\newtheorem{remark}[thm]{Remark}

\numberwithin{equation}{section}

\newcommand{\A}{\mathbb{A}}
\newcommand{\R}{\mathbb{R}} 
\newcommand{\G}{\mathbb{G}}
\newcommand{\N}{\mathbb{N}}
\newcommand{\C}{\mathbb{C}}

\newcommand{\mres}{%
  \,\raisebox{-.127ex}{\reflectbox{\rotatebox[origin=br]{-90}{$\lnot$}}}\,%
 }

\begin{document}

\title{ON DIMENSION AND REGULARITY OF BUNDLE MEASURES}

\author{R. Ayoush}
\thanks{R.A. was supported by the Warsaw Center of Mathematics and Computer Science}
\address{Institute of Mathematics, Polish Academy of Sciences
00-656 Warszawa, Poland}
\email{rayoush@impan.pl}

\author{M. Wojciechowski}
\address{Institute of Mathematics, Polish Academy of Sciences
00-656 Warszawa, Poland}
\email{m.wojciechowski@impan.pl}


\begin{abstract}
In this paper we quantify the notion of antisymmetry of the Fourier transform of certain vector valued measures. The introduced scale is related to the condition appearing in Uchiyama's theorem and is used to give a lower bound for the rectifiable dimension of those measures. Moreover, we obtain an estimate of the lower Hausdorff dimension assuming certain more restrictive version (in the structural sense) of the $2$-wave cone condition for PDE-constrainted measures, extending its applications to more general Fourier-analytic setting. The article contains also a theorem concerning regularity: we prove that elements of considered class vanish on 1-purely unrectifiable sets.
\end{abstract}

\maketitle

\section{Introduction}
\subsection{Formulation of results}
Geometric structure and dimensional properties of distributional gradients of functions from $BV(\R^{n})$  are well studied and widely applied (cf. \cite{A}, \cite{ACPP}, \cite{AFP}, \cite{PW}). It is known, for example, that their lower Hausdorff dimension is at least $n-1$ and that it is an optimal bound. Moreover, those measures cannot charge $(n-1)$-purely unrectifiable sets of finite $\mathcal{H}^{n-1}$ measure (see Lemma 3.76 and Theorem 3.78 in \cite{AFP}).
For the class of bundle measures, introduced in \cite{RW}, we can consider analogous problems. 

\begin{definition}
By $\G(m,E)$ let us denote the Grassmannian of $m$-dimensio\-nal subspaces of some fixed, $d$-dimensional real vector space $E$.
 We call a bundle any continuous function  $\phi:\mathbb{R}^{n}\backslash\{0\} \to \G(m,E)$. If additionally  $\phi(a\xi)=\phi(\xi)$ for any positive $a$,  then we refer to it as a homogeneous bundle.
\end{definition}

This setting gives a possibility to define bundle measures by imposing Fourier analytic rigidity conditions:

\begin{definition}
For any homogeneous bundle $\phi$, by $M_{\phi}(\mathbb{R}^{n},E)$ we denote the set of non-zero vector measures taking values in $E$ such that for each $\xi \neq 0$ there exists $c(\xi) \in \C$ such that $c(\xi) \hat{\mu}(\xi) \in \phi(\xi)$.
\end{definition}

The above definition generalizes the mentioned example of gradient measures. Indeed, if $f \in BV(\R^{n})$ then $\widehat{\nabla f }(\xi) = 2\pi i \xi \hat{f}(\xi)$, so $\nabla f \in M_{\phi}(\R^{n},\R^{n})$ for a particular bundle $\phi (\xi) = span\{\xi\}$.
As a matter of fact, we cannot make a use of the ideas from the classical proofs for $BV$ gradients due to the absence of sufficiently general coarea formula.

In this article, for the sake of simplicity, we limit our interest mostly to the case of line bundles ($m = 1$), for other ones all proofs can be adapted with straightforward modifications.  Above notation is complementary to the language of $A$-free measures (\cite{DR}); see  Chapter 4 for the explanation.

We propose a conjecture that links antisymmetry of a bundle with dimension of vector measures.
\begin{definition}
By the lower Hausdorff dimension of a non-zero (scalar or vector) measure $\mu$ we understand
$$
\dim_{H}(\mu) = \inf \{\alpha: \exists F \text{ - Borel set,} \ \mu(F) \neq 0, \ \dim_{H} F = \alpha \}.
$$
\end{definition}
\begin{definition} We say that a nonconstant line bundle $\phi$ is antisymmetric on $l$-dimensional subspheres or $l$-antisymmetric ($l=0,1,...,n-1$), if for each $(l+1)$-dimensional subspace  $V \subset \R^{n}$ there exist $\xi_{1}, \xi_{2} \in V \cap S^{n-1}$ such that $\phi(\xi_{1}) \neq  \phi(\xi_{2})$. Denote
$$
a(\phi) =\min \{l: \phi \ \ \text{is} \ l\text{-antisymmetric} \}.
$$
\end{definition}

\begin{con}
\label{ascon}
If $\mu$ is a bundle measure subordinated to a smooth, nonconstant bundle $\phi$, then
$$
 \dim_{H}(\mu) \geq n-a(\phi).
$$
\end{con}

 Our first result confirms correctness of Conjecture~\ref{ascon} under only mild geometric assumptions. To prove this we use classical measure-theoretic method of blowing-up measures, modified for dealing with Fourier transforms.

\begin{thm}
	\label{mainres}
	a) Suppose that $\mu$ is a bundle measure subordinated to a smooth, nonconstant bundle $\phi$. Then
	$$
	\dim_{rect}(\mu) \geq n-a(\phi),
	$$
	where
\begin{equation*}
\begin{split}
\dim_{rect}(\mu):=\min\{k: \exists k \text{-rectifiable  measure} \ \nu  \\ \text{s.t.} \ \mu_{\mres F} = \nu \neq 0\ \text{for some Borel set} \ F\}\cup\{n\}.
\end{split}
\end{equation*}
b) If $\mu$ is rectifiable then either
\[
dim_{H}(\mu) \geq \frac{n}{2}
\]
or $\mu$ can be identified with a scalar measure (its values belong to some line).
\end{thm} 

 Next, in Chapter 3, we prove two theorems which can be treated as extensions of the main result from \cite{RW} (Theorem \ref{rwthm}). Perhaps, the most signifficant theorem of our paper is the following:  
\begin{thm}\label{2dim}
 Suppose that $\mu$ is subordinated to a Lipschitz bundle $\phi$.
If there exist $2-$dimensional spaces $V_{1}, \dots, V_{J}$ such that $\phi(V_{i}\setminus \{0\})$ is contained in a linear space $W_{i}$ (of an arbitrary dimension) and $\cap_{i} W_{i} = \{0\}$, then
\[
\dim_{H}(\mu) \geq 2.
\]
\end{thm}
The above condition is related to the $k$-wave cone scale introduced in \cite{ADHR} (see also Example \ref{afreeex} for the discussion).
We prove also a rectifiability result which, together with Theorem \ref{rwthm} may be treated as an analogue of Federer-Volpert theorem (\cite{AFP}, Theorem 3.78., Proposition 3.92.).
\begin{thm}
	\label{regthm}
	Suppose that $\phi : \mathbb{R}^{n}\backslash\{0\} \to \G(1,E) $ is a nonconstant,  homogeneous bundle, H\"older with exponent $> \frac{1}{2}$. Then, for $\mu \in M_{\phi}(\R^{n}, E)$  and any $1$-purely unrectifiable set $F$ satisfying $\mathcal{H}^{1}(F) < \infty$ we have 
	$\mu_{\mres F} \equiv 0$.
\end{thm}
Proofs of both results are based on the theory of $s$-Riesz sets from \cite{RW}, extended to tempered Radon measures. To get the first one we combine it with a suitable use of Salem sets, inspired by so called one-sided Frostman lemma from \cite{KO}.  The rectifiability theorem is obtained by an application of Besicovitch-Federer theorem, which seems to be a new approach for this type of problems.

Chapter 4 contains examples and comparison with some known results about measures satisfying differential equations.

\subsection{Motivation and brief history of the problem}

Conjecture \ref{ascon} is inspired by Uchiyama's theorem on multiplier characterization of Hardy spaces (Theorem~\ref{uchithmquant}) which gives a proof when $a(\phi)=0$. It appeared while an attempt to answer a question from \cite{RW}:

\begin{con}(\cite{RW}, Conjecture 1)
	\label{rwcon}
	If the Fourier transform of a bundle measure $\mu$ contains $n$ linearly independent vectors and $\mu \in M_{\phi}(\mathbb{R}^{n}, E)$ for some line bundle $\phi$, then 
	$\dim_{H}(\mu) \geq n-1$.
\end{con}

\begin{thm}(\cite{RW}, Theorem 3)
	\label{rwthm}
	Let $\phi$ be a nonconstant line bundle, H\"older with exponent~$>~\frac{1}{2}$.  Then $dim_{H}(\mu) \geq 1$ for
	each $\mu \in M_{\phi}(\mathbb{R}^{n}, E)$.
\end{thm}

Theorem~\ref{rwthm} covers the case '$a(\phi) = n-1$' which is on the endpoint opposite to Uchiyama's theorem. In this paper we concentrate on the intermediate points of the scale. Conjecture~\ref{rwcon} was inspired by the example of measures derived from $BV$, that is, satisfying equation $\nabla f = \mu$ for some $f \in L^{1}(\R^{n})$ in the sense of distributions. This result shows, in particular, that if in such problem we replace $\nabla$ by any so called canceling operator (see \cite{V} and Example~\ref{Vex}), then the resulting measure has lower Hausdorff dimension at least 1.
Let us also mention that a particular case of the main result from \cite{SW} is a proof of the above conjecture for measures given by $(D_{1}^{s}f,...,D_{n}^{s}f) = \mu$  for some natural $s$ and $f \in L^{1}(\R^{n})$ ($\phi(\xi) = span \{(\xi_{1}^{s},\xi_{2}^{s}, \dots, \xi_{n}^{s})\}$). In this situation we have $a(\phi)=1$.

The technique used in \cite{SW} revealed strong connections of  dimension estimates with embedding theorems. Briefly: the better range of an embedding connected with a differential operator, the higher lower bound of dimension it gives. It is worth mentioning that canceling and elliptic operators (see \cite{V} or Example~\ref{Vex} for definitions) are precisely those which determine critical Sobolev embedding:

\begin{thm}
	\label{vanthm} (\cite{V}, Theorem 1.3.)  Suppose that $A(D)$ is a homogeneous differential operator of rank $s$ on $\R^{n}$ from $V$ to $W$. Then the estimate
	$$
	\lVert D^{s-1} f \rVert_{L^{\frac{n}{n-1}}} \leq C \lVert A(D)f \rVert_{L^{1}}
	$$
	holds for $f \in C^{\infty}_{c}(\R^{n};V)$ if and only if $A(D)$ is elliptic and canceling.
\end{thm}

Let us also underline that the theorem of Uchiyama gives the answer when the Hardy space $H^{1}(\R^{n})$ norm is equivalent to a norm given by a family of multipliers.

\begin{thm}
	\label{uchithmquant}
	Let $\theta_{1}(\xi), \dots , \theta_{n}(\xi) \in C^{\infty}(S^{n-1})$ and $K_{\theta_{i}}f =\mathcal{F}^{-1}(\theta_{i}(\frac{\xi}{|\xi|}) \mathcal{F}(f))$. Then the inequality
	$$
	\frac{1}{C} \lVert f \rVert_{H^{1}} \leq \sum_{i=1}^{n} \lVert K_{\theta_{i}} f \rVert_{L^{1}} \leq C \lVert f \rVert_{H^{1}}
	$$
	is true for some constant $C$ if and only if
	$$
	rank \begin{bmatrix}
	\theta_{1}(\xi) & \theta_{2}(\xi)  &  \dots &  \theta_{n}(\xi) \\
	\theta_{1}(-\xi) & \theta_{2}(-\xi)  &  \dots &  \theta_{n}(-\xi)
	\end{bmatrix} \equiv 2
	$$
	for  $\xi \in S^{n-1}$.
\end{thm}

The above remarks suggest that the mechanism of creating singularities and validity of some norm inequalities are governed by the same phenomenon.

Conjecture \ref{ascon} was proposed by the first author in his master's thesis \cite{Ay}. Theorems \ref{mainres} and \ref{regthm} appeared in an early preprint of this paper.
In the meantime, B. Raita (independently) in \cite{Ra} posed a question analogous to Conjecture \ref{ascon} for measures solving differential equations. For the same setting, article \cite{ADHR} yielded another estimates and rectifiability results in terms of other type of antisymmetry/cancelation. Condition appearing in Theorem \ref{2dim} (result later in time than \cite{ADHR}) is close to one point of a cancelation scale from \cite{ADHR} (Example \ref{afreeex}).

\subsection{Conventions}
Throughout the paper, we use the following notation: \\
$n$ – dimension of the ambient space $\R^{n}$, \\
$d$  –  dimension of $E$, i.e. space containing values of bundle measures, \\
$m$ – parameter of the Grassmanian, i.e. dimension of its elements,\\
$l$ – degree of antisymmetry/dimension of the wave cone.

While assuming Lipschitz or H\"older continuity of a bundle we mean a corresponding property of its restriction to the unit sphere. 

By $M(\mathbb{R}^{n})$ we understand the set of finite Radon measures. For $f \in L^{2}(\R^{n})$ and $\mu \in M(\R^{n})$ we choose the following normalization of the Fourier transform:
$$
\hat{f}(\xi) = \int_{\mathbb{R}^{n}}e^{-2\pi i \langle \xi,x \rangle}f(x)dx,
$$
$$
\hat{\mu}(\xi) = \int_{\mathbb{R}^{n}}e^{-2\pi i \langle \xi,x \rangle}d\mu(x).
$$

In the paper we use the below definition of rectifiability

\begin{definition} A set $E \subset \mathbb{R}^{n}$ is called $k$-rectifiable, if there exist Lipschitz functions $f_{i} : \mathbb{R}^{k} \to \mathbb{R}^{n}, i=1,2,...,$ such that
$$
 \mathcal{H}^{k} (E \setminus \bigcup^{\infty}_{i=1}f_{i}(\mathbb{R}^{k})) = 0 .
$$
A set $F \subset \R^{n}$ is called purely $k$-unrectifiable if $\mathcal{H}^{k}(F \cap E) = 0$ for every $k$-rectifiable $E$.
We call a (scalar or vector) measure $\mu$ $k$-rectifiable if there exist a $k$-rectifiable set $E$ and a Borel function  (scalar or vector) $f$ such that $\mu = f \mathcal{H}^{k}_{\mres E}$.
\end{definition}

For a vector space $V$ and a vector $u$ we denote $p_{V}, p_{u}$ orthogonal projections on $V$ and on $span\{u\}$ respectively.
A symbol $\mathcal{D}(\R^{n})$ means for us the space of smooth functions with compact support. 
By the spectrum of a tempered Radon measure we understand the support of its distributional Fourier transform. We denote it by $spec(\cdot)$.

 $W^{s,2}$ stands for the $L^{2}$-based Sobolev space of distributions, with order of smoothness $s \in \R$.

\section{Estimates for the rectifiable part}
\subsection{Tangent measures and rectifiability}
The notion of tangent measure (see \cite{P}) is extremally useful in geometric measure theory. However, one has to be careful while using it in Fourier analysis. For example, it is not hard to construct a measure whose one of tangent measures, in the classical sense, is not a tempered distribution (see also \cite{O} for a more pathological example). In this and the next subsection we present how to preserve Fourier analytic constraints in the limit, by modifying the definition of tangency.
\begin{definition} (\cite{P})
For a given $r>0$ and a Radon measure $\mu$ we define its blow-up by the formula $\mu_{r,x}(A)=\mu(x+rA)$. Any measure $\nu$ which is a weak-$\ast$ limit in 
$M(\R^{n})$ of a sequence of the type
$$
c_{i}\mu_{x,r_{i}}
$$
for some positive $\alpha$ and $r_{i} \downarrow 0$ and $c_{i}>0$ we call tangent measure to $\mu$ at point $x$. We denote the set of those measures by $Tan(\mu, x)$.
\end{definition}

The above definition can be easily extended to vector measures (in this case, convergence is understood as the weak-$\ast$ convergence of coordinates in $M(\R^{n})$). For rectifiable measures it suffices to consider normalizaitons of blow-ups given by suitable power functions.

\begin{definition}
For a fixed $\alpha>0$, by $Tan_{\alpha}(\mu,x)$ we denote the subset of $Tan(\mu,x)$ obtained by taking $c_{i} = r_{i}^{-\alpha}$. By $Tan^{*}(\mu,x)$ and $Tan^{*}_{\alpha}(\mu,x)$ we denote subsets of $Tan(\mu,x)$ and $Tan_{\alpha}(\mu,x)$, respectively, consisting of tempered Radon measures which are limits of blow-ups in $\mathcal{S}'(\R^{n})$.
\end{definition}
 
A straightforward generalization of Theorem 4.8 from \cite{D}  or Theorem 2.83 from \cite{AFP} is 
\begin{thm}
\label{rectchar}
Let $\mu =  f \mathcal{H}^{k}_{\mres E}$ be a $k$-rectifiable vector measure. Then for $ \mathcal{H}^{k}$-a.e. $x \in E$
there exists a $k-$dimensional vector space $V_{x}$ such that
$$
r^{-k}\mu_{x,r} \rightarrow f(x)\mathcal{H}^{k}_{\mres V_{x}}, 
$$ 
in the weak-$\ast$ topology as $r \downarrow  0$.
\end{thm}
\begin{cor}\label{decay}If $\mu \in M(\R^{n})$ is a $k$-rectifiable measure, then for $\mathcal{H}^{k}$-a.e. x there exists $C_{x}>0$ such that $\lvert \mu \rvert (B(x,r)) \leq C_{x}r^{k}$.
\end{cor}
Convergence from Theorem~\ref{rectchar}  is tested on functions from $C_{c}(\R^{n})$. However, for our applications we need convergence in $\mathcal{S}'(\R^{n})$. This requires extending the class of test functions to $\mathcal{S}(\R^{n})$ and can be achieved with the following lemma:
\begin{lem}
\label{convlem}
 Suppose that $\mu \in M(\R^{n})$ and $\alpha > 0$. Then: a) $Tan^{*}_{\alpha}(\mu,x) = Tan_{\alpha}(\mu,x) $.
 b) If $g \in L^{1}(\mu)$ then $Tan^{*}_{\alpha}(g\mu, x) = g(x)Tan^{*}_{\alpha}(\mu, x)$ for $\mu$-a.e. $x$.
\end{lem}
\begin{proof}
Let us notice first that b) is implied by a) and an analogous property of $Tan_{\alpha}(\mu,x)$ (see Proposition 3.12. in \cite{D}). 

To prove a) it suffices to show that if $r^{-\alpha}\mu_{x,r} \to \nu$ in $M(\R^{n})$, then also $r^{-\alpha}\mu_{x,r} \to \nu$ in $\mathcal{S}'(\R^{n})$ (reverse inclusion is obvious). Choose any $\varphi \in \mathcal{S}(\R^{n})$. We can write $\varphi = \sum_{i=1}^{\infty} \varphi_{i}$, where $\varphi_{i} \in C^{\infty}$, $supp(\varphi_{i}) \subset B(0,i)\setminus B(0,i-1)$ for $i>1$ and $supp(\varphi_{1}) \subset B(0,1)$. Moreover, we can assume that $\lVert \varphi_{i} \rVert_{\infty} \leq \lVert \varphi_{|B(0,i)\setminus B(0,i-1)} \rVert_{\infty}$. Then
$$
\bigg{\lvert}\frac{1}{r^{\alpha}} \int \varphi d\mu_{x,r} - \int \varphi d\nu \bigg{\rvert} \leq \bigg{\lvert}\frac{1}{r^{\alpha}} \int  \sum_{i=1}^{j} \varphi_{i} d\mu_{x,r} - \int  \sum_{i=1}^{j} \varphi_{i} d\nu \bigg{\rvert} +
$$
$$
 + \bigg{\lvert}\frac{1}{r^{\alpha}} \int  \sum_{i>j} \varphi_{i} d\mu_{x,r} \bigg{\rvert} + \int  \sum_{i>j} |\varphi_{i}| d|\nu|.
$$
Second term can be majorized by
$$
\sum_{i>j} \frac{\lVert \varphi_{i} \rVert_{\infty}|\mu|(B(x,ir))}{r^{\alpha}} \leq C\sum_{i>j}i^{\alpha} \lVert \varphi_{i} \rVert_{\infty} ,
$$
(we used Corollary~\ref{decay}) and the third one is a tail of a convergent series.
After taking sufficiently big $j$ and then choosing suitable $r_{0}$, we see that for $r<r_{0}$ the starting expression is smaller than
any a priori given positive number.
\end{proof}
\subsection{Distributional definition of bundle measures}
We say that a bundle $\phi:\mathbb{R}^{n} \setminus \{0\} \to \G(m,E)$ is $C^{\infty}$ if (locally)
$\phi(x) = span\{e_{1}(x),...,e_{m}(x)\}$, where $(e_{1}(x),...,e_{m}(x))$ is an orthonormal system and $e_{i}(x)$ are $C^{\infty}$ functions.
For a bundle $\phi$ we can define pointwise its orthogonal complement by $\phi^{\bot }(x):=\phi(x)^{\bot}$. 
Of course, if $\phi$ is $C^{\infty}$, then so is $\phi^{\bot}$ (one can see it while applying Gram-Schmidt  orthogonalization). In this section all bundles are $C^{\infty}$.

\begin{definition}
For a $C^{\infty}$-bundle $\phi$, by $\mathcal{S}_{\phi}(\mathbb{R}^{n})$ we denote the set of vector valued Schwartz functions $f$ such that $f(x) \in \phi(x)$ for $x \in \mathbb{R}^{n} \setminus \{ 0 \}$.
\end{definition}

\begin{definition}
By $\mathcal{S}'_{\phi}(\mathbb{R}^{n})$ we understand the class of vectors of tempered distributions $(\Lambda_{1},...,\Lambda_{d})$ ($d$ = $\dim E$) satisfying
$$
\sum_{i=1}^{d} \langle \hat{\Lambda_{i}},f_{i} \rangle = 0
$$
for an arbitrary $(f_{1},...,f_{d}) \in \mathcal{S}_{\phi^{\bot}}(\mathbb{R}^{n})$ . This is equivalent to
$$
\sum_{i=1}^{d} \langle \Lambda_{i},\hat{f_{i}} \rangle = 0.
$$
\end{definition}

Further we prove that this  class contains bundle measures and that it is preserved by taking limits of blow-up processes. We use Parseval's identity (see \cite{K}, p. 145):
\begin{thm}
\label{parseval}
  If $\mu \in M(\mathbb{R}^{n}, E)$ and $f \in \mathcal{S}(\mathbb{R}^{n})$, then
  $$
  \langle f,\mu \rangle = \int f(x) d\mu(x) =\int \hat{f}(\xi) \hat{\mu}(-\xi) d\xi.
  $$
\end{thm}

\begin{lem}
\label{conslem}
Let $\mu \in M_{\phi}(\mathbb{R}^{n},E)$. If at some point $x$ there exists a  tangent (vector) measure  $\nu \in Tan^{*}(\mu,x)$, then it belongs to $\mathcal{S}'_{\phi}(\mathbb{R}^{n})$.
\end{lem}
\begin{proof}

\textbf{Step 1.} We have $c_{k}\mu_{x,r_{k}} \in M_{\phi}(\mathbb{R}^{n},E)$: \\
Indeed, for a fixed coordinate $\mu^{(j)}$ we have
$$
c_{k}\widehat{\mu}_{x,r_{k}}^{(j)} (\xi) = c_{k} \int_{\mathbb{R}^{n}} e^{-2\pi i \langle \xi,\frac{y-x}{r_{k}} \rangle} d\mu^{(j)}(y) =
$$
$$
c_{k} e^{2\pi i \langle \xi, \frac{x}{r_{k}} \rangle} \int_{\mathbb{R}^{n}} e^{-2\pi i \langle \frac{\xi}{r_{k}}, y \rangle} d\mu^{(j)}(y) =
c_{k} e^{2\pi i \langle \xi, \frac{x}{r_{k}} \rangle} \widehat{\mu}^{(j)}\Big{(}\frac{\xi}{r_{k}}\Big{)}, 
$$
hence $c_{k}\widehat{\mu}_{x,r_{k}} (\xi) \parallel \widehat{\mu}(\xi) $.
\\

\textbf{Step 2.} If $\mu = (\mu_{1},...,\mu_{d}) \in M_{\phi}(\mathbb{R}^{n},E)$, then $\nu \in \mathcal{S}'_{\phi}(\mathbb{R}^{n})$: \\ Let $(f_{1},...,f_{d}) \in \mathcal{S}_{\phi^{\perp}}(\mathbb{R}^{n})$. By Parseval's identity we have

\[
\sum_{i=1}^{d} \langle \mu_{i},\hat{f_{i}} \rangle =  \int \sum_{i=1}^{d} f_{i}(\xi)\hat{\mu}_{i}(\xi) d\xi = 0 ,
\]
because $(f_{1},...,f_{d})(\xi)$ and $c(\xi)(\hat{\mu}_{1},...,\hat{\mu}_{d})(\xi)$ are perpendicular at each $\xi \neq 0$.

\textbf{Step 3.} Let $(f_{1},...,f_{d}) \in \mathcal{S}_{\phi^{\perp}}(\mathbb{R}^{n})$. Then
\[
0 = \lim_{r_{j} \downarrow 0} r_{j}^{-\alpha} \sum_{i=1}^{d} \langle \mu_{x,r_{j}}^{(i)}, \hat{f_{i}} \rangle = \sum_{i=1}^{d} \langle \nu^{(i)}, \hat{f_{i}}\rangle.
\]
\end{proof}

\subsection{Proof of Theorem \ref{mainres}}

We begin with invoking a well-known fact, whose proof can be found in \cite{H}  (Theorem 7.1.25).
\begin{lem}
\label{hmlem}
If $V \subset \mathbb{R}^{n}$ is a $k$-dimensional linear subspace, then $\mathcal{H}^{k}_{\mres V} \hat{ }=  \mathcal{H}^{n-k}_{\mres V^{\bot}}$.
\end{lem}

Now, by using Lemma~\ref{convlem}, we can reduce our considerations to the case of flat measures.

\begin{lem}
\label{csphlem}
Suppose that a measure $\mu \in M_{\phi}(\R^{n}, E )$ has a tangent measure (from $Tan^{*}$) of the form
$v\mathcal{H}^{k}_{\mres V}$,
where $\dim V = k$, and $	v$ is some non-zero vector. Then $\phi \equiv span \{v\}$  on $V^{\perp } \setminus \{0\}$.
\end{lem}
\begin{proof}
Let us take any vector-valued function $F \in \mathcal{S}_{\phi^{\perp}}(\mathbb{R}^{n})$. Then, by the preceeding lemma and the definition of $\mathcal{S}_{\phi^{\perp}}(\mathbb{R}^{n})$ we obtain:
$$
 \int_{\mathbb{R}^{n}} \langle F(x),v \rangle d\mathcal{H}^{n-k}_{\mres V^{\bot }}(x) = 0
$$ 
(here, brackets under integral sign mean standard scalar product in $\R^{d}$). Let us assume that at some $x_{0} \in V^{\perp } \setminus \{0\}$ we have $\phi(x_{0}) \neq span\{ v\}$. This implies the existence of $w \in \phi^{\bot}(x_{0})$ such that $\langle w, v \rangle \neq 0$, say  $\langle w, v \rangle > 0$. Take any function $g \in  \mathcal{S}_{\phi^{\perp}}(\mathbb{R}^{n})$ such that $g(x_{0}) = w$. Obviously, $\langle g(x), v \rangle > 0$ in some neighbourhood $U_{x_{0}}$ of $x_{0}$. After multiplying $g$ coordinatewise by a suitable mollifier supported at $U_{x_{0}}$ and substituting it in place of $F$ we get a contradiction.
\end{proof}

Our efforts may be summarized as follows:

\begin{lem}
\label{rectdim}
Suppose that $\mu$ is subordinated to $\phi$ and has at $x$ non-zero tangent measure (from $Tan^{*}$) of the form $v\mathcal{H}^{k}_{\mres V}$, where $\dim V = k$. Then we have
$$
k \geq n-a(\phi).
$$
\end{lem}
\begin{proof}
	By using Lemma~\ref{csphlem} we get $\dim V^{\perp} \leq a(\phi)$.
\end{proof}

Now we prove the main result of this chapter.

\begin{proof}\textit{of Theorem \ref{mainres}} a) Let $\nu$ and $F$ be such that $\mu_{\mres F} = \nu$ and $\nu$ is $k$-rectifiable. By Lemma~\ref{convlem} b) we can assume that $\mu=\nu$, since we can take $g = \chi_{F}$. Then, the unique tangent measure at a generic point $x$ is of the form $f(x) \mathcal{H}^{k}_{\mres V_{x}}$, where $f(x)$ is the density with respect to the Hausdorff measure and $V_{x}$ is the tangent plane to $\mu$ at $x$. The result follows from Lemma~\ref{rectdim}.
	
b) Let $\mu = f d\mathcal{H}^{k}_{\mres E}$ be such a measure and assume $k<\frac{n}{2}$. Observe that, by Theorem~\ref{rectchar}, Lemma~\ref{convlem} a) and Lemma~\ref{csphlem} for $\mathcal{H}^{k}_{\mres E}$-a.e. $x$ from the set $\{y: f(y) \neq 0\}$ we have $\phi \equiv span\{f(x)\}$ on $V^{\perp}_{x} \setminus \{0\}$. But $\dim V^{\perp}_{x} > \frac{n}{2}$, which means $V^{\perp}_{x} \cap V^{\perp}_{y} \neq \{0\}$ and consequently $span\{f(x)\} = span\{f(y)\}$ for any two such points. Hence, the density $f(x)$ is $\mathcal{H}^{k}_{\mres E}$ - a.e. parallel to some fixed vector, which shows that $\mu$ can be identified with a scalar measure.
\end{proof}

\section{Two extensions of Theorem \ref{rwthm}}
\subsection{Remarks on a theorem concerning $s$-Riesz sets}

\begin{definition}
	Subset $A\subset \R^{n}$ is called a Riesz set if $spec(\mu) \subset A$ implies that $\mu \in M(\R^{n})$ is absolutely continuous.
\end{definition}

\begin{definition}(\cite{RW})
	Subset $A\subset \R^{n}$ is called an $s$-Riesz set if $\dim_{H}(\mu) \geq s$ for each $\mu \in M(\R^{n})$ whose spectrum lies inside $A$.
\end{definition}

Next theorems give examples of Riesz sets.
\begin{thm}\label{riesz1}
	{\bf (F. and M. Riesz)} 
	 	If a measure $\mu \in M(\R)$ has its spectrum inside some half-line, then it is absolutely continuous with respect to the Lebesgue measure. 
\end{thm}

\begin{thm}\label{riesz2}
Suppose that a measure $\mu \in M(\R^{2})$ has its spectrum inside some angle of measure strictly smaller than $\pi$. Then it is absolutely continuous with respect to the full Lebesgue measure.
\end{thm}

Both theorems have its higher dimensional analogues; Theorem 0.3. from \cite{Ro} generalizes all mentioned above cases. To produce examples of $s$-Riesz sets, in \cite{RW} authors used the following slicing property.

\begin{thm}\label{sriesz1}(\cite{RW}, Theorem 1)
Let $A\subset \R^{n}$. If there exists a $k$-dimensional subspace $V \subset \R^{n}$ such that $\forall a\in\R^{n}$ $(V+a)\cap A$ is a Riesz set on $V+a$, then $A$ is a $k$-Riesz set.
\end{thm}

Since the argument in this theorem is based on the fact that orthogonal projections do not increase Hausdorff dimension, its easy modification gives a little bit more:

\begin{thm}\label{sriesz}
Let $A\subset \R^{n}$ and suppose that $\mu \in M(\R^{n})$ has its spectrum inside $A$. If there exists a $k$-dimensional subspace $V \subset \R^{n}$ such that $\forall a\in\R^{n}$ $(V+a)\cap A$ is a Riesz set on $V+a$, then $\mu(F)=0$ for each $F$ such that $\lambda_{V}(p_{V}(F))=0$ (where $\lambda_{V}$ is the Lebesgue measure on $V$).
\end{thm}

\begin{example}
	Bounded sets are Riesz sets. Indeed, let $A$ be a bounded set and let $f \in \mathcal{S}(\R^{n})$ be such that $\hat{f} \equiv 1$ on some ball containing $A$. Then we have an identity $\mu = \mu * f \in L^{1}(\R^{n})$ for any $\mu$ with spectrum inside $A$.
\end{example}

\begin{example}
Our model set is the following: let $V\subset \R^{n}$ be a $k$-dimensional subspace and let $f: \R \to \R$ be any strictly increasing function such that $\lim_{x\to +\infty} f(x)= +\infty$. In coordinates $\xi = (\xi_{1}, \xi_{2}) \in V \times V^{\perp}$ denote
\[
 B_{f} = \{(\xi_{1},\xi_{2}): |\xi_{1}| \geq 1, |\xi_{2}|\leq f(|\xi_{1}|)\},
\]
then $\R^{n} \setminus B_{f}$ is a $k$-Riesz set. This is a consequence of Theorem~\ref{sriesz}, previous example and the fact that slices of $\R^{n} \setminus B_{f}$ with planes parallel to $V$ are bounded.
\end{example}

Next we present a stronger version of Theorem~\ref{sriesz} for tempered measures. Namely, we allow $\hat{\mu}$ to be an $L^{2}$ function outside $s$-Riesz sets. In exchange, we require certain stability with respect to taking $\epsilon$-neighbourhoods.

\begin{cor}
\label{l2pert}
Let $A\subset \R^{n}$ and $\mu$ be a tempered Radon measure. Suppose that:
\begin{enumerate}
\item restriction (in the sense of distributions) of $\hat{\mu}$ to $\R^{n}\setminus A$ is an $L^{2}$ function,
\item there exists a $k$-dimensional subspace $V \subset \R^{n}$ such that for sufficiently small $\epsilon>0$, $\forall a\in\R^{n}$ $(V+a)\cap (A+B(0,\epsilon))$ is a Riesz set on $V+a$,
\end{enumerate}
then $\mu(F)=0$ for each $F$ such that $\lambda_{V}(p_{V}(F))=0$.
\end{cor}

\begin{proof}
	Suppose that there exists a bounded set $F$ contradicting the thesis. Assume first that $\hat{\mu} = 0$ outside $A$.
	For any $\delta > 0$ we can find a function $f \in  \mathcal{S}(\mathbb{R}^{n})$ such that $\hat{f} \in \mathcal{D}(\mathbb{R}^{n})$ and $|f(x)-1|< \delta$ for $x \in F$.
	
	\textit{Construction:} Take $g \in \mathcal{D}(\mathbb{R}^{n})$ such that $\int g = 1$ and denote $f = \check{g}$.
	Then $f(0) = 1$ and there exists $U$, a neighbourhood of $0$, such that $|f(x)-1|< \delta$ for $x \in U$. Of course 
	$\forall_{r>0} f(\frac{x}{r})\hat  \ \in \mathcal{D}(\mathbb{R}^{n})$. Taking big $r$ such that $F \subset rU$ we get a suitable function. Also, for sufficiently large $r$, $spec(f)$ is contained in arbitrarily small ball.
	
	Denote $\nu = f d\mu$.
	For sufficiently small $\delta$, $\nu(F) \neq 0$, $\nu$ is a finite measure and $spec(\nu) \subset spec(\mu) + spec(f)$ ($\nu$ is a product of a tempered distribution $\mu$ and a Schwartz function $f$). Hence, the spectrum of $\nu$ is as in the Theorem~\ref{sriesz}, which gives a contradiction.
	
	Now, if $\hat{\mu} = h \neq 0$ outside $A$ for some $h \in L^2$, then it suffices to apply previous reasoning for $\mu - \check{h}$ (changing $\mu$ by absolutely continuous measures has no impact on singular sets).
\end{proof}
\begin{remark}
Sets $B_{f}$ clearly satisfy assumption $(2)$ of Corollary~\ref{l2pert}.
\end{remark}
Let us go further and ask what can be said if restriction of $\hat{\mu}$ to $\R^{n} \setminus A$ is close to an $L^{2}$ function in some sense? For example if it is a Fourier transform of a distribution  from the fractional Sobolev space $W^{-s,2}$? If the negaitve order of smoothness $-s$ may be taken arbitrarily close to zero, then the lower bound of the Hausdorff dimension remains the same (though this trade-off formally costs us expected result about projections). This answer is obtained by the following lemma which employs a technique used in \cite{KO} and involves using properties of Salem sets.

\begin{lem}\label{tradeoff} Let $\mu \in M(\R^{n})$ and $F$ be a Borel set such that $\dim_{H}(F) = \alpha$ and $\mu(F) \neq 0$. Then, for any $0 < \eta \leq 1$ there exists a probability measure on $\R^{n}$ satisfying the following properties:
	\begin{enumerate}[a)]
		\item $|\hat{\nu}(\xi)| \lesssim |\xi|^{-\frac{\eta}{4}}$,
		\item $\nu$ is supported on a compact set $G$ s.t. $\dim_{M}(G) \leq 2\eta$,
		\item there exists $\tilde{F}$ such that $\mu*\nu(\tilde{F}) \neq 0$ and $\dim_{H}(\tilde{F}) \leq \alpha + 2\eta$.
	\end{enumerate}
\end{lem}
\begin{proof}To get first two properties it suffices to consider an image of a uniform measure on $\eta$-dimensional Cantor subset on $\R$ by the $n$-dimensional Brownian motion. Theorem 12.1. from \cite{M2} or Theorem 1 from Chapter 17 in \cite{Kah} gives a), while b) is implied by a well known fact that trajectories of the Brownian motion are almost surely $\beta$-H\"older continuous with $0<\beta<\frac{1}{2}$.
	
	Now let us prove  c). For simplicity, suppose that $\mu$ is real-valued and $\mu(F)>0$. By regularity and Jordan decomposition theorem for measures, we may assume that $F$ is compact and, for some $\delta>0$, its $\delta$-neighbourhood $F_{\delta}$ satisfies $\mu_{-}(F_{\delta})< \frac{1}{100} \mu(F)$. It suffices to rescale previously obtained $\nu$ so that $G \subset B(0, \frac{\delta}{2})$ and take $\tilde{F} = F + G$. Indeed
	\[
	\mu*\nu(F+G) = \int_{G} \mu(F+G-x) d\nu(x)
	\]
	and $F \subset F+G-x \subset F_{\delta}$ for any $x \in G$, so the integral is positive. Moreover, $\dim_{H}(F+G) \leq \dim_{H}(F)+\dim_{M}(G) \leq \alpha + 2\eta$, (\cite{KO}, Lemma 2) which proves the lemma.
\end{proof}

The above immediately leads to the announced corollary:

\begin{cor}
	\label{sobolevpert}
	Let $A\subset \R^{n}$ and $\mu$ be a tempered Radon measure. Suppose that:
	\begin{enumerate}
		\item for an arbitrary $s>0$ restriction (in the sense of distribution) of $\hat{\mu}$ to $\R^{n}\setminus A$ is a Fourier transform of an element of $W^{-s,2}$,
		\item there exists a $k$-dimensional subspace $V \subset \R^{n}$ such that for sufficiently small $\epsilon>0$ $\forall a\in\R^{n}$ $(V+a)\cap (A+B(0,\epsilon))$ is a Riesz set on $V+a$,
	\end{enumerate}
	then $dim_{H}(\mu) \geq k$.
\end{cor}
\begin{proof}
Suppose that $\mu(F) \neq 0$, and $\eta>0$ is such that $dim_{H}(F)+2\eta < k$. For this $\eta$, take $\nu$ from Lemma \ref{tradeoff} and convolve it with $\mu$. Then, the restriction of $\widehat{\mu*\nu}$ to $\R^{n} \setminus A$ is in $L^{2}$, but (c) from Lemma \ref{tradeoff} and Corollary \ref{l2pert} give a contradiction.
\end{proof}
\subsection{Proof of Theorem \ref{2dim}}
Before the proof let us recall that in fact we assume Lipschitz continuity of the restriction of $\phi$ to the unit sphere. We use the standard metric on $\G(m,E)$, that is
 \[d_{\G(m,E)}(V, W) = \sup_{z\in V\cap S^{d-1}} dist_{E}(z,W).
 \]
\begin{proof}
Let us assume that  for some $A \subset \mathbb{R}^{n}$ such that  $\dim_{H}(A)<2$, we have $\mu(A)=e \neq 0$, and take $j$ satisfying $e \notin W_{j}$. There exists a functional $\theta \in E^{*}$ satisfying $W_{j} \subset ker \ \theta$ and $\theta(e) \neq 0$. Its value on $v$ may be computed as follows: project $v$ on $span \{e\}$ along a subspace containing $W_{j}$ (but not $e$) and take scalar product with $e$. Let $\nu \in M(\mathbb{R}^{n})$ 
be defined by the formula $\nu = \theta(\mu)$. Then we have $\hat{\nu} = \theta(\hat{\mu}(\xi))$, \boldmath $\nu(A) \neq 0$ \unboldmath and for some constant $C=C(ker \ \theta, e)$
the following estimate holds
\begin{equation}\label{dist}
|\hat{\nu}(\xi)| \leq C|\hat{\mu}(\xi)|\cdot |sin \angle(\phi(\xi), W_{j})| = C|\hat{\mu}(\xi)| \cdot dist_{\G(1,E)}(\phi(\xi),W_{j}) \leq
\end{equation}
\[
\leq C\lVert \mu \rVert \cdot dist_{\G(1,E)}(\phi(\xi),W_{j})
\]
This is obvious if $e$ is orthogonal to $W_{j}$ (we can take $C=1$). If not, we use the fact that two functionals with the same kernel are proportional.

In coordinates $\xi = (\xi_{1}, \xi_{2}) \in V_{j} \times V_{j}^{\perp}$ let $B_{f}$ be given by
\[
  B_{f} = \{(\xi_{1},\xi_{2}): |\xi_{1}| \geq 1, |\xi_{2}|\leq f(|\xi_{1}|)\},
\]
where $f(t)=\log(1+t)$. Then $\R^{n} \setminus B_{f}$ is a 2-Riesz set and we can apply Corollary \ref{sobolevpert}. Indeed, we will show that the distribution $(\hat{\nu} 1_{B_{f}})\check{}$ \ is in $W^{-s,2}$ for an arbitrary $s>0$.

Lipschitz continuity of a bundle easily implies that we have an inequality \\ $dist_{\G(1,E)}(\phi(z),W_{j}) \lesssim dist_{\R^{n}}(z,V_{j})$ for $z \in S^{n-1}$. Indeed, let $z_{0} \in V_{j}\cap S^{n-1}$ be such that $dist_{\R^{n}}(z, z_{0}) = dist_{\R^{n}}(z, V_{j}\cap S^{n-1})$. Then 
\begin{equation}\label{lip}
dist_{\G(1,E)}(\phi(z),W_{j}) \leq dist_{\G(1,E)}(\phi(z),\phi(z_{0})) \lesssim dist_{\R^{n}}(z, z_{0}) \simeq dist_{\R^{n}}(z, V_{j}).
\end{equation}
By homogenity and above remarks we obtain
\[
\int_{B_{f}} |\hat{\nu}(\xi)|^{2}|\xi|^{-2s}d\xi \lesssim
\]
\begin{multline*}
 \stackrel{\scriptscriptstyle\eqref{dist}}{\lesssim} 
 \int_{\{1\leq |\xi_{1}| < \infty\}}\int_{\{ |\xi_{2}| \leq f(|\xi_{1}|) \}}dist_{\G(1,E)}\Big{(}\phi\Big{(}\frac{\xi}{|\xi|}\Big{)},W_{j}\Big{)}^{2} |\xi|^{-2s} d\xi_{2}d\xi_{1} 
\end{multline*}
\[
\stackrel{\scriptscriptstyle\eqref{lip}}{\lesssim} \int_{\{1\leq |\xi_{1}| < \infty\}}\int_{\{ |\xi_{2}| \leq f(|\xi_{1}|) \}}dist_{\R^{n}}\Big{(}\frac{\xi}{|\xi|},V_{j}\Big{)}^{2} |\xi|^{-2s} d\xi_{2}d\xi_{1}
\]
\[
\lesssim \int_{\{1\leq |\xi_{1}| < \infty\}} f(|\xi_{1}|)^{(n-2)}\bigg{(}\frac{f(|\xi_{1}|)}{|\xi_{1}|}\bigg{)}^{2}|\xi_{1}|^{-2s} d\xi_{1}
\]

For an arbitrary $\gamma>0$, $f(u) \leq C_{\gamma}u^{\gamma}$ when $u\geq 1$, so the last integral may be majorized, up to a constant, by 
\begin{equation*}
\begin{split}
\int_{1 \leq |\xi_{1}| < \infty} |\xi_{1}|^{\gamma n-2-2s} d\xi_{1} = 2\pi \int_{1}^{\infty} t^{\gamma n-1-2s} dt,
\end{split}
\end{equation*}
which is finite for $\gamma < \frac{2s}{n}$. Choosing such $\gamma$ we obtain $\lVert (\hat{\nu} 1_{B_{f}})\check{} \ \rVert_{W^{-s,2}} < \infty$, so by Corollary~\ref{sobolevpert} we get \boldmath$\nu(A)=0$\unboldmath, which gives a contradiction.
\end{proof}

\begin{remark}
	The proof works if we assume Lipschitz continuity of $\phi$ at points from $\cup_{i} V_{i}\cap S^{n-1}$ only.
\end{remark}

\subsection{Rectifiability of bundle measures}
As we have seen in the proof of Theorem~\ref{mainres}, the homogenity condition gives us a possibility to relate geometry of singular sets with values of bundles measures. Proof of Theorem \ref{regthm}, which employs similar principles, is a consequence of a qualitative reformulation of Theorem~\ref{rwthm}.

\begin{definition}
	For $A \subset \mathbb{R}^{n}$, by $N(A)$ we denote $\{v \in \mathbb{R}^{n}: \lVert v \rVert = 1, \ \lambda(p_{v}(A)) = 0 \}$. Here, for  $v \in \mathbb{R}^{n} \setminus \{ 0\}$, by $p_{v}:\mathbb{R}^{n} \to span\{v\}$ we understand the orthogonal projection on $span\{v\}$ and $\lambda$ stands for the $1$-dimensional Lebesgue measure on $span\{v\}$. 
\end{definition}

\begin{thm}
	\label{edthm}
	If $\phi : \mathbb{R}^{n}\backslash\{0\} \to \G(1,E) $ is a  homogeneous bundle, H\"older with exponent $> \frac{1}{2}$, then for each $\mu \in M_{\phi}(\mathbb{R}^{n},E)$ and an arbitrary Borel set $A \subset \mathbb{R}^{n}$ we have
	\[
	N(A) \subset \phi^{-1}(\mu(A)) :=  \{ u \in \mathbb{R}^{n}: \lVert u \rVert = 1, \ \mu(A) \in \phi(u) \}.
	\]
\end{thm}

Note, that it proves Conjecture~\ref{ascon} if we replace $\dim_{H}$ by the lowest dimension of an affine subspace on which a measure does not vanish.

\begin{proof}\textit{(Sketch)} Let $A \subset \mathbb{R}^{n}$, $\mu(A)=e$, $\lambda(p_{v}(A)) = 0$ and assume that the thesis does not hold, i.e. $v \notin \phi^{-1}(e)$ for some $v$. We can choose a functional $\theta \in E^{*}$ satisfying $\phi(v) \subset ker \ \theta$ and $\theta(e)\neq 0$. The rest of the proof goes similarly as in Theorem \ref{2dim}. We replace $V_{j}$ by $span\{v\}$ and, by making a use of H\"older continuity, we prove that $\hat{\mu}$ is square summable in $B_{f}$. Instead of Corollary \ref{sobolevpert} we invoke Corollary \ref{l2pert}.
\end{proof}
To prove Theorem \ref{regthm}, we will need a part of Besicovitch-Federer projection theorem (see Theorem 18.1 in  \cite{M}):

\begin{thm} \label{bfederer} Let $A\subset \R^{n}$ be a Borel set with $\mathcal{H}^{m}(A) < \infty$, where $m<n$ is an integer. Then, $A$ is purely $m$-unrectifiable if and only if $\mathcal{H}^{m}(p_{V}(A)) = 0$ for almost all $V \in \G(m,\R^{n})$ (with respect to the natural measure on the Grassmannian).
\end{thm}

\begin{proof}\textit{(of Theorem \ref{regthm})}
	Suppose that $\mu(F') \neq 0$ for some $F' \subset F$. Then, by Theorem~\ref{edthm}, $N(F') \subset \phi^{-1}(\mu(F'))$. But, by Besicovitch-Federer  theorem $N(F')$   is a dense subset of $S^{n-1}$ so, by the continuity of the bundle, $\phi$ is constant, which gives a contradiction.
\end{proof}

\begin{remark}
Because in the proof of Corollary \ref{sobolevpert} we modified a measure by convolving it with a Salem measure, analogous method does not give rectifiability in Theorem \ref{2dim}.
\end{remark}

\begin{remark}
Theorems \ref{2dim} and \ref{regthm} can be proved for more general class than homogeneous bundles. For example, since addition of square-summable functions do not have any influence on singular sets of measures, we may admit certain error in the sense of $L^{2}$ norm (c.f. original formulation of Theorem \ref{rwthm} in \cite{RW}).
\end{remark}

\section{Connections with PDE-constrainted measures}

\begin{example} If $(D_{1}^{s}f,D_{2}^{s}f,D_{3}^{s}f) = \mu$  for some natural $s$ and $f \in L^{1}(\R^{3})$, then we have $\hat{\mu}(\xi) = (2\pi i)^{s} (\xi_{1}^{s},\xi_{2}^{s},\xi_{3}^{s})\hat{f}(\xi)$, so $\phi(\xi) =  span \{(\xi_{1}^{s},\xi_{2}^{s},\xi_{3}^{s})\}$ and $E = \R^{3}$. Let us take \[
	V_{1}=span\{e_{2}, e_{3}\}, V_{2}=span\{e_{1}, e_{3}\}, V_{3}=span\{e_{1}, e_{2}\}.
\] Then $\phi(V_{j}\setminus \{0\}) = V_{j}$ and $\cap_{i}^{3}  V_{i} = \{0\}$, so assumptions of Theorem \ref{2dim} are fulfilled. In particular, we obtained a purely Fourier-analytic proof of dimension estimate for gradients from $BV(\R^{3})$.
\end{example}

\begin{example}\label{Vex}(cf.\cite{V})
	Let $V,W$ be some finitely dimensional vector spaces, $n \geq 1$ and $s \in \N$. Suppose that $A(D)$ is a homogeneous differential operator of order $s$ on $\R^{n}$ from $V$ to $W$, that is
	\[
	A(D)u = \sum_{\alpha \in \N^{n}, |\alpha| = s} A_{\alpha}(\partial^{\alpha}u)
	\]
	for $u \in C^{\infty}(\R^{n}, V)$, where $A_{\alpha} \in \mathcal{L}(V,W)$. We say that $A(D)$ is canceling if
	\[
	\bigcap_{\xi \in \R^{n} \setminus \{0\}} \A(\xi)[V] = \{0\},
	\]
	where $\A(\xi)$ stands for the principal symbol of $A$. Assume, that  in the above $\dim \ V = 1$ and for some $f \in L^{1} (\R^{n})$ we have $A(D)f = \mu$ in the sense of distributions. If $A$ is elliptic ($A(\xi)\neq 0$ for $\xi \neq 0$), then the canceling condition means that $\mu$ is subordinated to a nonconstant bundle $\phi(\xi)=span\{\A(\xi)\}$.
	
\end{example}

In the case of general bundles with values in $\G(m,E)$, the $l$-antisymmetry condition can be formulated as follows:
\textit{for each $(l+1)$-dimensional subspace  $V \subset \R^{n}$ there exist $\xi_{1},...,\xi_{j} \in V \cap S^{n-1}$ such that $\phi(\xi_{1})\cap ... \cap \phi(\xi_{j}) = \{0\}$}.

\begin{example}\label{vsmulti}(Continuation of Example~\ref{Vex})
Assume that columns of a matrix $\A(\xi)$: $\A_{1}(\xi),...,\A_{m}(\xi)$ are linearly independent. Then, the operator $A$ satisfies the canceling condition if and only if the bundle $\phi(\xi) = span\{\A_{1}(\xi),...,\A_{m}(\xi)\}$ is $(n-1)$-antisymmetric.
\end{example}

\begin{example}\label{afreeex}(\cite{ADHR}, Theorem 1.3., Corollary 1.4.)
Suppose that for $A(D)$ as before we have
\[
A(D)\mu = 0
\]
in the weak sense. If $\A(\xi)$ has constant rank, then any such measure belongs to the class given by the bundle $\phi(\xi) = ker\{\A(\xi)\}$ and the $l$-wave cone condition from \cite{ADHR} reads as $\cap_{U \in \G(l,V)} \phi(U \setminus \{0\})=\{0\}$. In the mentioned paper it is proved, among other things, that under this assumption, any such measure is at least $l$-dimensional. 
Hence, in this setting a constraint in Theorem \ref{2dim} is a particular case of the $2$-wave cone condition. However, we do not require any connections with differential operators or even smoothness of the bundle. In fact, our proof requires only Lipschitz continuity of $\phi$ at points from $\cup_{i}V_{i}\cap S^{n-1}$.
\end{example}

\section{Acknowledgements}
The authors would like to thank B. Kirchheim, F. Nazarov, B. Raita and D. Stolyarov for valuable discussions.

\end{document}